\documentclass[12pt,reqno]{amsart}
\usepackage{amssymb,amsmath}
\def\({\left(}
\def\){\right)}
\def\Nx{\nabla}

\def\eb{\varepsilon}
\def\Cal{\mathcal}
\def\eb{\varepsilon}

\def\rw{\rightarrow}
\def\Om{\Omega}
\def\om{\omega}

\def\la{\lambda}

\def\R {\mathbb{R}}

\def\F {{\mathcal F}}

\def\Z {{\mathbb Z}}

\def \p {\partial}
\def\px{\partial_x}
\def\pt{\partial_t}
\def \ptu  {\partial_t u}

\def \and{\qquad\text{and}\qquad}

\def \px {\partial_x}
\def\Dx{\Delta}
\newcommand{\be}{\begin{equation} }
\newcommand{\ee}{\end{equation} }

\newtheorem{proposition}{Proposition}[section]
\newtheorem{theorem}[proposition]{Theorem}

\newtheorem{lemma}[proposition]{Lemma}
\theoremstyle{definition}

\newtheorem{remark}[proposition]{Remark}

\numberwithin{equation}{section}

\def \no#1#2#3 {{\bf #1} (#3), #2.}
\def \eds#1#2#3 {#1, #2, #3.}

\title[Feedback Control]
{Finite-Parameters Feedback Control for Stabilizing Damped Nonlinear Wave Equations}
\author[]
{Varga K. Kalantarov and Edriss S. Titi}

\address{(V.K.Kalantarov) Department of mathematics, Ko{\c c} University,
\newline\indent Rumelifeneri Yolu, Sariyer 34450\newline\indent
Sariyer, Istanbul, Turkey} \email{vkalantarov@ku.edu.tr}

\address{(E.S.Titi) Department of Mathematics, Texas AM University, 3368 TAMU,
\newline\indent College Station, TX 77843-3368, USA.  {\bf ALSO},
\newline\indent Department of Computer Science and Applied Mathematics, Weizmann
\newline\indent Institute of Science, Rehovot 76100, Israel.}
\email{titi@math.tamu.edu}
  \email{edriss.titi@weizmann.ac.il}

 
\begin{document}

\begin{abstract}{In this paper we introduce  a finite-parameters feedback control algorithm for stabilizing solutions of various classes of
damped nonlinear wave equations. Specifically, stabilization the zero steady state solution  of initial boundary value problems for nonlinear weakly and
strongly damped wave equations, nonlinear wave equation with
nonlinear damping term and some related nonlinear wave equations, introducing a feedback control terms that employ parameters, such as,  finitely many Fourier modes, finitely many volume elements and
finitely many nodal observables and controllers. In addition, we also establish the stabilization of the zero steady state solution to initial boundary value problem for the damped
nonlinear wave equation with a controller acting in a proper subdomain. Notably, the feedback controllers proposed here can be equally applied for stabilizing other solutions of the underlying equations.}
\end{abstract}
\date{January 3, 2015}
\maketitle

\section{Introduction}\label{s0}

This paper is devoted to the study of finite-parameters feedback control for stabilizing solutions of initial boundary value problems for nonlinear damped wave equations.
Feedback control, stabilization and control of wave equations
is a well established area of control theory.
Many interesting results were obtained, in the last dacades, on stabilization of linear and nonlinear wave equations (see , e.g., \cite{BaSl}, \cite{Har}, \cite{Kom},\cite{LaTr},\cite{Liu},\cite{Mar},\cite{PeVaZu}   and references therein). Most of the problems considered
are problems of stabilization by interior and boundary controllers involving linear or nonlinear damping terms.
 However, only few results are known on feedback stabilization  of linear hyperbolic equation by employing finite-dimensional controllers (see, e.g., \cite{Ba} and references therein).\\
We study the problem of feedback stabilization of initial boundary value problem for
damped nonlinear wave equation  \be\label{iw1} \pt^2u- \Dx u +b \ptu
-a u+f(u)= -\mu w, \ x \in \Om, t>0, \ee
nonlinear wave equation
with nonlinear damping term
 \be\label{iw1a} \pt^2u- \Dx u +b g(\ptu)
-a u+f(u)= -\mu w, \ x \in \Om, t>0, \ee
and the strongly damped wave equation
 \be\label{iw2} \pt^2 u- \Dx u -b\Dx \ptu
-\la u+f(u)= -\mu w, \ x \in \Om, t>0.
 \ee

Here and in what follows   $\mu,a,b,\nu $ are given positive
parameters, $w$ is a feedback control input (different for different problems),
 $f(\cdot): \R\rw \R$ is a given continuously differentiable function such that $f(0)=0$ and

\be\label{F} f(s)s-\F(s)\geq 0, f'(s)\geq 0, \ \forall s \in \R, \ \
\F(s):=\int_0^sf(\tau)d\tau, \ \ \forall s\in \R.
\ee

We show that the feedback controller proposed in \cite{AzTi} for nonlinear parabolic equations can
be used for stabilization of above mentioned wide class of nonlinear
dissipative wave equations. We also show stabilization of the initial boundary value problem for equation \eqref{iw1},
with the control input acting on some proper subdomain of $\Om$.\\
Our study, as well as the results obtained in \cite{AzTi}, are inspired by the fact that dissipative dynamical systems generated by initial boundary value problems, such as the 2D
Navier - Stokes equations, nonlinear reaction-diffusion equation, Cahn-Hilliard equation,
damped nonlinear Schr\"{o}dinger equation, damped nonlinear Klein - Gordon equation, nonlinear strongly damped wave equation and related equations and systems
have a finite dimensional asymptotic (in time) behavior (see, e.g., \cite{BaVi},\cite{Ch}- \cite{CDT96},\cite{FoPr}-\cite{FoTi},\cite{La1},\cite{La2},\cite{Te} and references therein). This property has also been implicitly used in the feedback control of the Navier-Stokes equations overcoming the spillover phenomenon in \cite{Cao-Kevrekidis-Titi}\\
Motivated by this observation, we specifically  show\\
\begin{enumerate}
\item  stabilization of  $1D$ damped wave equation
\eqref{iw1}, under Neuman boundary conditions, when the feedback control input
employs observables based on measurement of {\it finite volume elements},\\

\item stabilization of  equation \eqref{iw1}, under
homogeneous Dirichlet boundary condition, with one feedback
controller supported on some proper subdomain of $\Om\subset \R^n$,\\

\item  stabilization of  equation \eqref{iw1} and equation \eqref{iw2}, under
the Dirichlet boundary condition, when the feedback control involves
{\it finitely many Fourier modes of the solution}, based on eigenfunctions of the Laplacian
subject to homogeneous Dirichlet boundary condition,

\item stabilization of the $1D$ equation \eqref{iw2},
when the feedback control incorporates observables at {\it finitely
many nodal points}.
\end{enumerate}

In the sequel we will use the notations:\\
\begin{itemize}
\item $(\cdot,\cdot)$ and $\|\cdot\|$ denote the inner product and
the norm of $L^2(\Om)$,\\
and the following inequalities:\\
\item Young's inequality
\be\label{yng}
ab\leq \eb a^2+\frac \eb 4 b^2,
\ee
that is valid for all positive numbers $a,b$ and $\eb$
\item the Poincar\'{e} inequality
\be\label{pnk}
\|u\|^2\leq \la_1^{-1}\|\Nx u\|^2,
\ee
which holds for each $u\in H_0^1(\Om)$.
\end{itemize}

\section{Feedback control of damped nonlinear wave equations}

In this section, we show that the initial boundary value problem for nonlinear damped wave equation can be stabilized
by employing finite volume elements feedback controller, feedback controllers acting in a subdomain of $\Om$, or feedback controllers involving finitely many Fourier modes.\\

{\bf 1.  Stabilization employing  finite volume elements feedback control.} We consider the
following feedback control problem
\begin{equation}\label{kg1}
\begin{cases}
\pt^2u- \nu \px^2{u} +b \ptu -\la u+f(u)=
-\mu\sum\limits_{k=1}^N\overline{u}_k\chi_{J_k}(x), \ \ x \in (0,L), \ t>0, \\
\px{u}(0,t)=\px{u}(L,t)=0, \ t>0,\\
u(x,0=u_0(x), \ \ \ptu (x,0)=u_1(x), \ \ x\in (0,L).
\end{cases}
\end{equation}

Here $J_k:=\left[(k-1)\frac LN, k\frac LN\right),$ for $k=1,2,\cdots N-1$ and $J_N=[\frac{N-1}{N}L, L]$,\\
$\overline{\phi}_k:=\frac1{|J_k|}\int\limits_{J_k}\phi(x)dx$, and
$\chi_{J_k}(x)$ is the characteristic function of the interval
$J_k$. In what follows we will need the following lemma

\begin{lemma}\label{pr}(see \cite{AzTi})  Let $\phi\in H^1(0,L)$.  Then
\be\label{p1}
\|\phi-\sum\limits_{k=1}^N\overline{\phi}_k\chi_{J_k}(\cdot)\|\leq h
\|\phi_x\|, \ee and \be\label{p2} \|\phi\|^2\leq
h\sum\limits_{k=1}^N\overline{\phi}_k^2+\left(\frac{h}{2\pi}\right)^2\|\phi_x\|^2,
\ee where $h:=\frac LN.$
\end{lemma}
By employing this Lemma, we proved the following theorem:
\begin{theorem}\label{t1} Suppose that the nonlinear term $f(\cdot)$ satisfies
the condition \eqref{F} and that $\mu$ and $N$ are large enough satisfying
\be\label{Nm} \mu \geq 2\left(\la +\frac{\delta_0b}{2}\right) \
\mbox{and} \ \
N^2>\frac{L^2}{2\nu\pi^2}\left(\la +\frac{\delta_0b}2\right),
\ee
where
 \be\label{del0}\delta_0=\frac{b}2\min\{1,\nu\}.\ee
Then each solution of the problem \eqref{kg1} satisfies the following decay estimate:
\be\label{Est1} \|\ptu(t)\|^2+\|\px{u}(t)\|^2\leq K(\|u_1\|^2+\|\px
u_0\|^2)e^{-\delta_0 t}, \ee where $K$ is some
positive constant depending on $b, \la$ and $L$.
\end{theorem}
\begin{proof} First observe that one can use standard tools of the
theory of nonlinear wave equations to show global existence and
uniqueness of solution to problem \eqref{kg1} (see, e.g., \cite{Lions}).\\
Taking the $L^2(0,L)$ inner product  of \eqref{kg1} with $\ptu+\eb u$  , where $\eb>0$ is a parameter, to be determined later, gives us
the following relation:
\begin{multline}\label{kg2}\frac d{dt}\left[\frac12\|\ptu\|^2+\frac\nu 2\|\px{u}\|^2-
\frac12(\eb b -\la)\|u\|^2+\int_0^L\F(u)dx+
\frac12h\mu\sum\limits_{k=1}^N\overline{u}_k^2+\eb(u,\ptu)
\right]\\
+(b-\eb)\|\ptu\|^2+\eb\nu\|\px{u}\|^2-\eb
\la\|u\|^2+\eb(f(u),u)+\eb\mu
h\sum\limits_{k=1}^N\overline{u}_k^2=0.
\end{multline}
It follows from \eqref{kg2} that
\begin{multline}\label{kg3}
\frac d{dt} \Phi_\eb(t)+\delta \Phi_\eb(t)+ (b-\eb)\|
\ptu\|^2+\eb\nu\|\px{u}\|^2-\eb \la\|u\|^2+\eb(f(u),u)\\
+ \eb\mu
h\sum\limits_{k=1}^N\overline{u}_k^2 -\frac{\delta}2\|\ptu\|^2-
\frac{\delta\nu}2\|\px{u}\|^2-\frac\delta2(\eb b-\la)\|u\|^2-\delta
(\F(u),1)\\
-\frac\delta2h\mu \sum\limits_{k=1}^N\overline{u}_k^2-\delta
\eb(u,\ptu)=0.
\end{multline}
Here
\begin{multline*}
\Phi_\eb(t):=\frac12\|\ptu(t)\|^2+\frac\nu2\|\px{u}(t)\|^2+\frac12(\eb b
-\la)\|u(t)\|^2+\int_0^L\F(u(x,t))dx\\
+
\frac12h\mu\sum\limits_{k=1}^N\overline{u}_k^2(t)+\eb(u(t),\ptu(t)).
\end{multline*}
Due to condition \eqref{F} and the Cauchy-Schwarz inequality  we have the
following lower estimate for $\Phi_\eb(t)$
$$
\Phi_\eb(t)\geq \frac14\|\ptu\|^2+\frac\nu2\|\px
u\|^2+(\frac{b\eb}2  -\frac\la2-\eb^2)\|u\|^2+\frac12h\mu
\sum\limits_{k=1}^N\overline{u}_k^2.
$$

 By choosing  $\eb \in (0,\frac{b}2]$
 and by employing inequality \eqref{p2}, we get from the above inequality:
\begin{multline}\label{kg4}
\Phi_\eb(t)\geq \frac14\|\ptu\|^2+\frac\nu2\|\px
u\|^2-\la\left[h\sum\limits_{k=1}^N\overline{u}_k^2+\left(\frac{h}{2\pi}\right)^2\|
\px{u}\|^2\right]\\
+
\frac12\mu h \sum\limits_{k=1}^N\overline{u}_k^2
=\frac14\|\ptu\|^2+\left(\frac\nu2 -\la\frac{L^2}{4\pi^2N^2}\right)\|\px  u\|^2+h(\frac\mu2-\la)\sum\limits_{k=1}^N\overline{u}_k^2,
\end{multline}
Thanks to the condition \eqref{Nm} we also have
\begin{equation}\label{ph}
\Phi_\eb(t)\geq \frac14\|\ptu(t)\|^2+d_0\|\px
u(t)\|^2,
\end{equation}
where $d_0=\frac{\delta_0bL^2}{4N^2\nu\pi^2}$, and  $\delta_0$ is defined in \eqref{del0}.\\
Let $\delta \in (0,\eb)$, to be chosen below. According to condition \eqref{F} $f(u)u\geq \F(u), \forall u\in \R$. Therefore
 \eqref{kg3} implies
\begin{multline}\label{kg5}
\frac d{dt} \Phi_\eb(t)+\delta \Phi_\eb(t)+
\frac b2\|
\ptu\|^2+(\frac {b\nu}2-\frac\delta2)\|\px{u}\|^2\\
+\left(\frac{\delta\la}{2}-\frac{\delta b^2}4-\frac{b\la}2\right)\|u\|^2
+\mu h\left(\frac b2-\frac\delta2\right)\sum\limits_{k=1}^N\overline{u}_k^2\leq0.
\end{multline}
We choose here $\delta =\delta_0:=\frac{b}2\min\{\nu,1\}$
 and employ inequality \eqref{p2} to obtain:
\begin{multline*}
\frac d{dt} \Phi_\eb(t)+\delta_0 \Phi_\eb(t)+
\frac b2\|
\ptu\|^2+\left[\frac {b\nu}4-\frac b2\left(\la+\frac{\delta_0b}{2}\right)\frac{L^2}{4\pi^2N^2}\right]\|\px{u}\|^2\\
+\frac{bh}2\left[\frac{\mu}{2}-\left(\la+\frac{\delta_0b}{2}\right)\right]\sum\limits_{k=1}^N\overline{u}_k^2\leq0.
\end{multline*}

Finally  by using condition \eqref{Nm} we get:
$$
\frac d{dt} \Phi_\eb(t)+\delta_0 \Phi_\eb(t)\leq 0.
$$
 Thus by Gronwalls inequality and thanks   \eqref{ph} we have
\be\label{kg5} \|\ptu(t)\|^2+\|\px{u}(t)\|^2\leq K( \|u_1\|^2+\|\px
u_0\|^2)e^{-\delta_0 t}, \ee
where $\delta_0$ is defined in \eqref{del0}, and $K$ is a positive constant, depending on $b,\la$ and $L$.\vspace{0.5cm}
\end{proof}
 \subsection{ Stabilization with  feedback control on a subdomain.}

 In this section we study the problem of internal stabilization of initial boundary value problem for nonlinear damped wave equation on a bounded domain.
We show that the problem can be exponentially stabilized by a feedback controller acting on a strict subdomain. So, we consider the following feedback control problem:
\be\label{ow1}
\pt^2u-\Dx u+b\ptu-a u +|u|^{p-2}u= -\mu \chi_{\om}(x)u, \ x \in \Om,
t>0, \ee
\be\label{ow2} u(x,t)=0, \ x \in \p\Om, \ t>0,
\ee
\be\label{ow3} u(x,0)=u_0(x), \ \ptu(x,0)=u_1(x), \ x \in \Om.
\ee
Here  $a>0,p\geq 2$ are given numbers, and $\mu >0$ is a parameter to be determined, $\Om\subset \R^n$ is a bounded domain with smooth boundary $\p \Om$,
$\chi_{\om}(x)$ is the characteristic function of the  subdomain $\om
\subset \Om$ with smooth boundary and $\overline{\om} \subset \Om$.\\
 Let
us denote by  $\la_1(\mu,\Om)$  the first eigenvalue of the elliptic problem
$$
-\Dx v +\mu \chi_{\om}(x) v =\la v; \ x \in \Om, \ v=0, \ x \in \p
\Om,
$$
and denote by  $\la_1(\Om_\om)$  the first eigenvalue of the problem
$$
-\Dx v  =\la v, \ x \in \Om_\om; \ v=0, \ x \in \p \Om_\om,
$$
where $\Om_\om:=\Om \setminus \overline{\om}.$
We will need the following Lemma in the proof of the  main result of this section:
\begin{lemma}\label{lem2}
(see, e.g.,\cite{XiYo}) For each $d>0$ there exists a number
$\mu_0(d)>0$ such that the following inequality holds true
\be\label{eig1} \int_\Om\left(|\Nx v(x)|^2+\mu
\chi_\om(x)v^2(x)\right)dx\geq \left(\la_1(\Om_\om)-d\right)\int_\Om
v^2(x)dx, \ \forall v \in H^1_0(\Om),\ee whenever $\mu >\mu_0$.
\end{lemma}
\begin{theorem}\label{intt}
Suppose that
\be\label{int1}
\la_1(\Om_\om)\geq 4a+\frac{3b^2}2, \ \ \mbox{and} \ \ \mu >\mu_0,
\ee
where $\mu_0$ is the parameter stated  in Lemma \ref{lem2} , corresponding to $d=\frac12\la_1(\Om_\om)$.
Then the energy norm of each weak solution of the problem \eqref{ow1}-\eqref{ow3}
 tends to zero with an exponential rate. More precisely, the following estimate holds true:
\be\label{Enest}
\|\ptu(t)\|^2+\|\Nx u(t)\|^2+\|u(t)\|_{L^p(\Om)}^p\leq C_0 e^{-\frac b2
t},
\ee
where $C_0$ is a positive constant depending on initial data.
\end{theorem}
\begin{proof}
Taking the $L^2(\Om)$ inner product of \eqref{ow1} with $\ptu+\frac b2 u$ we get
\be\label{ow5} \frac d{dt} E_b(t)+ \frac b2 \left[\|\ptu\|^2 +
\|\Nx u\|^2-a \|u\|^2+ \|u\|_{L^p(\Om)}^p+\mu
\int_{\Om}\chi_\om(x)u^2(x)dx\right]=0,
 \ee
where
\begin{multline*}
E_b(t):=\frac12\|\ptu\|^2 +\frac12\|\Nx u\|^2-\frac a2\|u\|^2+\frac1p
\|u\|_{L^p(\Om)}^p+\frac\mu2\int_{\Om}\chi_\om
(x)u^2(x)dx\\
+\frac b2(u,\ptu)+\frac{b^2}{4}\|u\|^2.
\end{multline*}
By Cauchy-Schwarz and Young inequalities we have
\be\label{csin}
\frac b 2|(u,\ptu)|\leq \frac14\|\ptu\|^2+\frac{b^2}4\|u\|^2.
\ee
By employing \eqref{csin}, we obtain the lower estimate for $E_b(t)$:
\be\label{Eint}
E_b(t)\geq \frac14\|\ptu\|^2 +\frac14\|\Nx u\|^2+\frac1p
\|u\|_{L^p(\Om)}^p+\frac14\left[\|\Nx u\|^2+2\mu \|u\|^2_{L^2(\om)}-2a \|u\|^2\right].
\ee
According to Lemma \ref{lem2}
\be\label{pin1}
\|\Nx u\|^2+\mu  \|u\|^2_{L^2(\om)}\geq \frac12\la_1(\Om_\om)\| u\|^2, \ \ \mbox{for every} \ \  \mu >\mu_0.
\ee
Thus, thanks to condition \eqref{int1}, we have
\be\label{int2}
E_b(t)\geq \frac14\|\ptu\|^2 +\frac14\|\Nx u\|^2+\frac1p
\|u\|_{L^p(\Om)}^p.
\ee
Adding to the left-hand side of \eqref{ow5} the expression $E_b(t)-\delta E_b(t)$ with some $\delta>0$ (to be chosen below),
we get
\begin{multline*}
\frac d{dt} E_b(t)+ \delta E_b(t)+\frac12(b-\delta)\|\ptu\|^2 +\frac12(b-\delta)\|\Nx {u}\|^2
\\
+\frac12(a\delta-ab-\frac{\delta b^2}2)\|u\|^2+\frac{\mu}2(b-\delta)\|u\|^2_{L^2(\om)}-\frac12\delta b(u,\ptu)=0
\end{multline*}
We use here  the inequality
$$
\frac12\delta b|(u,\ptu)|\leq \frac b4\|\ptu\|^2+\frac{b\delta^2}{4}\|u\|^2,
$$
then in the resulting inequality we choose $\delta=\frac b2$, and get:
$$
\frac d{dt} E_b(t)+ \frac b2 E_b(t)+\frac b4\left[\|\Nx {u}\|^2+\mu\|u\|^2_{L^2(\om)}-(2a+\frac34b^2)\|u\|^2 \right].
$$

Finally, by using the condition \eqref{int1}, thanks to Lemma \ref{lem2} we obtain
$$
\frac d{dt} E_b(t)+\frac  b2E_b(t)\leq 0.
$$
Integrating the last inequality and taking into account \eqref{int2}, we
arrive at the desired estimate \eqref{Enest}.

 \end{proof}

 \vspace{0.5cm} {\bf 3. Stabilization employing finitely
many Fourier modes feedback controls.}
In this section we consider the feedback control problem for damped nonlinear wave equation  based on finitely many Fourier modes,
 i.e.  we consider the  feedback system of the following form:
\be\label{fo1} \pt^2u-\nu \Dx u +b\ptu-au+|u|^{p-2}u=-\mu
\sum_{k=1}^N(u,w_k)w_k, \ x \in \Om, t>0, \ee
\be\label{fo2} u=0, \
x \in \p \Om, \ t>0, \ee
\be\label{fo2a} u(x,0=u_0(x), \ \pt u(x,0)=u_1(x), \ \
x \in \Om, \ t>0. \ee
Here $\nu>0,a>0,b>0,\mu>0,p\geq 2$ are
given numbers;  $w_1,w_2,..., w_n,...$ is the set of orthonormal (in
$L^2(\Om)$) eigenfunctions  of the Laplace operator $-\Dx$ under the
homogeneous Dirichlet boundary condition, corresponding to
eigenvalues $0<\la_1\leq \la_2 \cdots \leq \la_n\leq, \cdots$.

\begin{theorem}\label{TF} Suppose that $\mu$ and $N$  are large enough such that
\be\label{cF1} \nu \geq (2a+3b^2/4)\la_{N+1}^{-1}, \ \ \mbox{and} \ \  \mu \geq a+3b^2/4. \ee
 Then the following decay estimate holds true
  \be\label{estW}
\|\ptu(t)\|^2+\|\Nx u(t)\|^2+\int_\Om |u(x,t)|^pdx\leq E_0e^{-\frac b2
t}, \ee
where
$$
E_0:=\frac12\|u_1\|^2+\frac \nu 2\|\Nx u_0\|^2\\
+(\frac{b^2}4-\frac
a2)\|u_0\|^2+ \frac1p\int_\Om |u_0(x)|^pdx
+\frac\mu2\sum_{k=1}^N(u_0,w_k)^2+\frac { b}2(u_0,u_1).
$$
\end{theorem}
\begin{proof}
Multiplication of
\eqref{fo1} by $\ptu+\frac b2 u$  and integration over $\Om$ gives
\begin{equation}\label{fo3}
\frac d{dt} E_b(t)+ \frac b2\left[ \|\ptu\|^2+  \nu \|\Nx
u\|^2-a\|u(t)\|^2\\
+\int_{\Om}|u|^pdx+\mu \sum_{k=1}^N(u,w_k)^2\right]=0,
\end{equation} where
\begin{multline*}
E_b(t):=\frac12\|\ptu(t)\|^2+\frac \nu 2\|\Nx u(t)\|^2\\
+(\frac{b^2}4-\frac
a2)\|u(t)\|^2+ \frac1p\int_\Om |u(x,t)|^pdx
+\frac\mu2\sum_{k=1}^N(u(t),w_k)^2+\frac { b}2(u(t),\ptu(t)).
\end{multline*}
Thanks to the Cauchy-Schwarz and Young inequalities we have
$$
\frac { b}2|(u(t),\ptu(t))|\leq \frac14\|\ptu\|^2+\frac {b^2}4\|u\|^2.
$$
Consequently,
$$
E_b(t)\geq \frac14\|\ptu\|^2+\frac \nu 2\|\Nx{u}\|^2-\frac a2\|u\|^2+
\frac1p\int_\Om |u|^pdx +\frac\mu2\sum_{k=1}^N(u,w_k)^2.
$$

Since
$$
-\frac a2 \|u\|^2+\frac\mu2\sum_{k=1}^N(u,w_k)^2=\frac12(\mu-a)
\sum_{k=1}^N(u,w_k)^2-\frac a2\sum_{k=N+1}^\infty(u,w_k)^2,
$$
by using  the Poincar\'{e} -like inequality
 \be\label{QN}
\|\phi-\sum_{k=1}^N(\phi,w_k)w_k\|^2\leq \la_{N+1}^{-1}\|\Nx
\phi\|^2, \ee
 which is valid for each $\phi\in
H^1_0(\Om)$, we get
 \be\label{fo4} E_b(t)\geq
\frac14\|\ptu\|^2+\frac12\left(\nu-a\la_{N+1}^{-1}\right)\|\Nx
u\|^2+\frac12(\mu -a)\sum_{k=1}^N(u,w_k)^2+\frac1p\int_\Om |u|^pdx.
\ee
Thus,
\be\label{ph1}
 E_b(t)\geq \frac14\|\ptu\|^2+\frac\nu 4\|\Nx{u}\|^2+
\frac1p\int_\Om |u|^pdx.
\ee

Adding to the left-hand side of \eqref{fo3} the expression $\delta E_b(t)-\delta E_b(t)$ (with $\delta$ to be chosen later), we obtain

\begin{multline*}
\frac d{dt} E_b(t)+ \delta E_b(t) +\frac12(b-\delta ) \|\ptu\|^2+\frac \nu2(b-\delta)\|\Nx u\|^2\\
+\left(-\frac{ba}2+\frac{\delta a}2-\frac{\delta b^2}4\right)\|u^2\|+\left(\frac b2-\frac \delta2\right)\|u\|^p_{L^p(\Om)}\\
+\frac \mu2(b-\delta)\sum\limits_{k=1}^N(u,w_k)^2-\frac12b\delta(u,\ptu)=0.
\end{multline*}
We choose here $\delta =\frac b2$, and obtain
\begin{multline}\label{bbb}
\frac d{dt} E_b(t)+ \delta E_b(t) +\frac b4 \|\ptu\|^2+\frac {\nu b}4\|\Nx u\|^2\\
-\left(\frac{ba}4+\frac{ b^3}8\right)\|u\|^2
+\frac {\mu b}4\sum\limits_{k=1}^N(u,w_k)^2-\frac{b^2}4(u,\ptu)=0.
\end{multline}
Employing the inequality
$$
\frac{b^2}4|(u,\ptu)|\leq \frac b4\|\ptu\|^2+\frac{b^3}{16}\|u\|^2,
$$
and inequality \eqref{QN}, we obtain from \eqref{bbb}:
\begin{multline}\label{bbb}
\frac d{dt} E_b(t)+ \delta E_b(t) +\frac { b}4\left[\nu -(a+3b^2/4)\la_{N+1}^{-1}\right]\|\Nx u\|^2\\
+\frac { b}4\left[\mu -(a+3b^2/4)\right]\sum\limits_{k=1}^N(u,w_k)^2.
\end{multline}

\begin{multline}\label{del3}
\frac d{dt} E_b(t)+ \delta E_b(t)+\left(\frac b2-\frac \delta2\right)\|\ptu\|^2+\left(\frac b2(\nu-a\la_{N+1}^{-1})- \frac{\delta \nu}2\right)\|\Nx u\|^2+\\
\left(\frac b2(\mu-a)-\frac{\delta \mu}2\right)\sum_{k=1}^N(u,w_k)^2+\left(\frac b2-\frac\delta p\right)\int_{\Om}|u|^pdx\\
-\delta\left(\frac{b^2}4-\frac a2\right)\|u\|^2
+\frac{\delta b}2(u,\ptu)\leq 0.
\end{multline}
By using the  inequalities
$$
\frac{\delta b}2|(u,\ptu)|\leq \frac b4\|\ptu\|^2+\frac{\delta^2 b}{4}\|u\|^2,
$$
and \eqref{QN} in \eqref{del3} we get
\begin{multline*}
\frac d{dt} E_b(t)+ \delta E_b(t)+\left(\frac b4-\frac \delta2\right)\|\ptu\|^2+\left(\frac b2(\nu-a\la_{N+1}^{-1})- \frac{\delta^2b}4\la_{N+1}^{-1}-\frac{\delta \nu}2\right)\|\Nx u\|^2+\\
\left(\frac b2(\mu-a)-\frac{\delta \mu}2\right)\sum_{k=1}^N(u,w_k)^2+\left(\frac b2-\frac\delta p\right)\int_{\Om}|u|^pdx-\delta\left(\frac{b^2}4-\frac a2-\frac{\delta^2 b}4\right)\|u\|^2
\leq 0.
\end{multline*}

By choosing  $\delta=\frac b2$ we obtain :
$$
\frac d{dt} E_b(t)+ \frac b2 E_b(t)+\frac b2\left[\nu -(\frac{b^2}8
+\frac 32a)\la_{N+1}^{-1}\right]\|\Nx
u\|^2+
\frac b2\left[\frac \mu 2-\frac32a-\frac{b^2}{8}\right]\sum_{k=1}^N(u,w_k)^2\leq 0.
$$
Taking into account  conditions \eqref{cF1} we  deduce from the last
inequality the inequality
$$
\frac d{dt} E_b(t)+ \frac b2 E_b(t)\leq 0.
$$
Integrating the last inequality we get the desired estimate \eqref{estW} thanks to \eqref{ph1}.
\end{proof}

\begin{remark}
We would like to note that  estimate \eqref{estW} allows us  to prove
existence of a weak solution to the problem \eqref{fo1}, \eqref{fo2}
such that (see \cite{Lions})
$$
u\in L^{\infty}\left(\R^+;H^1_0(\Om)\cap L^{p}(\Om)\right), \ u\in
L^{\infty}\left(\R^+; L^{2}(\Om)\right)
$$
Note that there are no restrictions on the spatial dimension of the domain $\Om$ or
the growth of nonlinearity.
\end{remark}

\section{Nonlinear Wave Equation with Nonlinear Damping term:Stabilization with finitely many Fourier
modes}
In this section we consider the initial boundary value problem for a nonlinear wave equation with nonlinear damping term with a
feedback controller involving finitely many Fourier modes:

\be\label{nfo1} \partial_t^2u-\nu \Dx u +b|\ptu|^{m-2}\ptu-au+|u|^{p-2}u=-\mu
\sum_{k=1}^N(u,w_k)w_k, \ x \in \Om, t>0, \ee
\be\label{nfo2} u=0, \
x \in \p \Om, \ t>0,
\ee
\be\label{nfo2a} u(x,0)=u_0(x), \ \ \ptu(x,0)=u_1(x), \ \
x \in \Om, \
\ee
where $\nu>0,b>0,a>0, p\geq m>2$ are given
parameters.\\
We show stabilization of solutions of this problem  to the zero  with a polynomial rate. Our main result is the following theorem:
\begin{theorem}\label{nTF} Suppose that $\mu$ ad $N$ are large  enough such that
\be\label{ncF1} \nu >2a\la_{N+1}^{-1}, \ \ \mbox{and} \ \ \mu
>a. \ee
 Then  for each solution of the problem \eqref{nfo1}-\eqref{nfo2a} following estimate holds true
\be\label{nestW}
\|\ptu(t)\|^2+\|\Nx u(t)\|^2+\int_\Om |u(x,t)|^pdx\leq C
t^{-\frac{m-1}{m}},
 \ee
where $C$ is a positive constant depending on initial data.
\end{theorem}
\begin{proof}
Taking the inner product of equation \eqref{nfo1} in $L^2(\Om)$ with $\ptu$ we get
\be\label{nfo3}
\frac d{dt} \Cal E(t)+  b\|\ptu(t)\|^m=0,
\ee
where
\be\label{nd1}
\Cal E(t):=\frac12\|\ptu(t)\|^2 +\frac12\|\Nx u(t)\|^2-\frac
a2\|u(t)\|^2+\frac1p \|u(t)\|^p_{L^p(\Om)}+\frac\mu2\sum_{k=1}^N(u(t),w_k)^2.
 \ee
Similar to \eqref{fo4} we have
\begin{multline}\label{nd2}  \Cal E(t)\geq
\frac12\|\ptu\|^2+\frac12\left(\nu-a\la_{N+1}^{-1}\right)\|\Nx
u\|^2+\frac12(\mu -a)\sum_{k=1}^N(u,w_k)^2+\frac1p\int_\Om |u|^pdx\\
\geq \frac12\|\ptu\|^2+\frac\nu2\|\Nx
u\|^2+\frac1p \|u(t)\|^p_{L^p(\Om)}.
\end{multline} This estimate implies that the function $\Cal E(t)$ is
non-negative, for $t\geq 0$. Let us integrate \eqref{nfo3} over the
interval $(0,t)$:
\be\label{nd3} \Cal E(0)-\Cal
E(t)=\int_0^t\|\ptu(\tau)\|_m^md\tau. \ee
Taking the inner product of
\eqref{nfo1} in $L^2(\Om)$ with $u$ gives
$$
\frac d{dt}(u,\ptu)=
\|\ptu\|^2-\nu\|\Nx u\|^2-b\int_\Om
u|\ptu|^{m-2}\ptu dx+a\|u\|^2-\int_\Om|u|^pdx-\mu
\sum\limits_{k=1}^N(u,w_k)^2.
$$
By using notation \eqref{nd1} we can rewrite the last relation
in the following form:
\begin{multline}\label{nd4} \frac
d{dt}(u,\ptu)=\frac32\|\ptu\|^2 -\Cal E(t)-\frac \nu2\|\Nx
u\|^2\\
+\frac
a2\|u\|^2-\frac\mu2\sum\limits_{k=1}^N(u,w_k)^2-\frac{p-1}{p}\int_\Om|u|^pdx-
b\int_\Om u|\ptu|^{m-2}\ptu dx. \end{multline}
By using inequality \eqref{QN} we obtain from \eqref{nd4}
\begin{multline*}
\frac d{dt}(u,\ptu)\leq - \Cal
E(t)+\frac32\|\ptu\|^2-\frac12(\mu-a)\sum\limits_{k=1}^N(u,w_k)^2\\
+\frac a2\sum\limits_{k=N+1}^\infty(u,w_k)^2-\frac\nu2\|\Nx
u\|^2+b\int_\Om|u||\ptu|^{m-1}dx \leq - \Cal
E(t)+\frac32\|\ptu\|^2\\
-\frac12(\mu-a)\sum\limits_{k=1}^N(u,w_k)^2-\frac12(\nu-a\lambda_{N+1}^{-1})\|\Nx
u\|^2+b\int_\Om|u||\ptu|^{m-1}dx.
\end{multline*}
Taking into account condition \eqref{ncF1} we deduce the
inequality
$$
\Cal E(t)\leq -\frac d{dt}
(u(t),\ptu(t))+\frac32\|\ptu(t)\|^2+b\int_\Om|u(x,t)|\ptu(x,t)|^{m-1}dx.
$$
After integration over the interval $(0,t)$, we obtain
\begin{multline}\label{nd5} \int_0^t\Cal E(\tau)d\tau\leq (u(0),\ptu(0))-(u(t),
\ptu(t))\\
+ \frac32\int_0^t\|\ptu(\tau)\|^2d\tau
+b\int_0^t\int_\Om|u(x,\tau)|\ptu(x,\tau)|^{m-1}dxd\tau.
\end{multline}
Since $\Cal E(t)\leq \Cal E(0)$ and $p\geq 2$, then due to \eqref{nd2}, we have
\be\label{nde1}
|(u(0),\ptu(0))-(u(t), \ptu(t))|\leq C,
 \ee
where $C$ depends on the initial data.\\
 By using
the H\"{o}lder inequality and estimate \eqref{nd3}, we estimate the
second term on the right-hand side of \eqref{nd5}:
\begin{multline}\label{nde2}
\int_0^t\|\ptu(\tau)\|^2d\tau
=\int_0^t\int_\Om|\ptu(x,\tau)|^2\cdot1dxd\tau\\
\leq
\left(\int_0^t\|\ptu(\tau)\|_{L^m(\Om)}^md\tau\right)^{\frac2m}(|\Om|t)^{\frac
{m-2}m}\leq Ct^{\frac {m-2}m}.
\end{multline}

The third term on the right-hand side of \eqref{nd5} we estimate again by using
the estimates \eqref{nd3} similarly (recalling that $m\leq p$):
\begin{multline}\label{nde3} \int_0^t\int_\Om |\ptu(x,\tau)|^{m-1}|u(x,\tau)|dx
d\tau\\
\leq \left( \int_0^t\int_\Om|\ptu(x,\tau)|^mdx
d\tau\right)^{\frac{m-1}m} \left( \int_0^t\int_\Om|u(x,\tau)|^mdx
d\tau\right)^{\frac{1}m}\leq Ct^{\frac1m} .
\end{multline}
Since $\Cal E(t)$ is non-increasing,  positive function we have:
 \be\label{nde4}
t\Cal E(t)\leq \int_0^t \Cal E(\tau)d\tau.
\ee
 Thus employing \eqref{nde1}-\eqref{nde4} we obtain from \eqref{nd5}
 $$
\Cal E(t)\leq C t^{-\frac{m-1}m}.
 $$
 Hence
 \be\label{ndest}
 \|\ptu(t)\|^2+\|\Nx u(t)\|^2+\|u(t)\|_{L^p(\Om)}^p\leq C t^{-\frac{m-1}m}.
 \ee
\end{proof}
\begin{remark}
We would like to note that unlike the result on finite  set of functionals determining long time behavior of solutions
 to equation
  $$
 \partial_t^2u-\nu \Dx u +b|\ptu|^{m-2}\ptu-au+|u|^{p-2}u=0, \ x \in \Om, t>0,
  $$
under the homogeneous Dirichlet's boundary condition, established in \cite{ChKa},
(as in the case of the equation \eqref{fo1}), we do not require restrictions neither on the dimension of the domain $\Om$ 
nor on the parameters $m>0, p>0$. It suffices to know that  problem \eqref{fo1}-\eqref{fo2a} has a global solution such that (see, e.g., \cite{Lions}):
$$
u\in L^{\infty}\left(\R^+;H^1_0(\Om)\cap L^{p}(\Om)\right)\cap L^{m+2}(\R^+; L^{m+2}(\Om)), \ u\in
L^{\infty}\left(\R^+; L^{2}(\Om)\right).
$$

\end{remark}

\section{Nonlinear Strongly Damped Wave equation}
In this section, we study the problem of feedback control of initial boundary
value problem for nonlinear strongly damped equation with controllers involving finitely many Fourier modes and by nodal observables.

{\bf 1. Feedback control employing finitely many nodal valued observables.}

First we consider first the following problem
\be\label{st1} \pt^2u-\px^2u-b\p_x^2\p_t u-a u+ f(u)=-\mu
\sum_{k=1}^Nhu(\bar{x}_k)\delta (x-x_k), \ x \in (0,L), t>0, \ee
\be\label{st2} u(0,t)=u(L,t)=0, \ x \in (0,L), t>0, \ee where
$x_k,\bar{x}_k \in J_k=[(k-1)\frac LN, k \frac LN],
k=1,...,N$,$h=\frac LN$, $f(\cdot)$ is continuously differentiable
function that satisfies the conditions \eqref{F}, $\delta(x-x_k)$ is
the Dirac delta function, $a,b$ and $\mu$ are given positive
parameters.

Our estimates will be based on the following lemma:
\begin{lemma}\label{lem3} (see,e.g., \cite{AzTi})
Let $x_k, \overline{x}_k  \in J_k=[(k-1) \, h, k \, h], k=1,..,N, $
where $h=\frac{L}{N}$, $N \in \Z ^{+}.$ Then for every $\varphi \,
\in \, H^1(0,L)$ the following inequalities hold true
\be\label{L1}
\sum_{k=1}^{N} | \varphi(x_k)- \varphi(\overline{x}_k)|^ 2 \leq \, h
\, \|\varphi_x\|^2_{L^2}, \ee and \be\label{L2} \|\varphi\|^2 \leq
\, 2 \, \left[h \, \sum_{k=1}^{N} |\varphi(x_k)|^2+ h^2 \,
\|\varphi_x\|^2 \right].
\ee
\end{lemma}
Taking the $H^{-1}$ action of \eqref{st1} on $\ptu+\eb u \in H^1$, where $\eb>0$, to be determined, we
get
\begin{multline}\label{st3}
\frac d{dt}\left[\frac12\|\ptu\|^2+\frac12(1+\eb
b)\|\px{u}\|^2-a\|u\|^2+(F(u),1)+\eb(u,\ptu)\right]+\\
b\|\p_t\p_x u\|^2+\eb \|\px{u}\|+\eb(f(u),u)-a \eb\|u\|^2-\eb\|\ptu\|^2=\\
-\mu h\sum_{k=1}^{N}u(\bar{x}_k)\ptu(x_k)-\eb\mu h
\sum_{k=1}^{N}u(\bar{x}_k)u(x_k).
\end{multline}
By using the equalities
$$
\sum_{k=1}^{N}u(\bar{x}_k)\ptu(x_k)=\frac12 \frac
d{dt}\sum_{k=1}^{N}u^2(\bar{x}_k)+\sum_{k=1}^{N}
u(\bar{x}_k)\left(\ptu(x_k)-\ptu(\bar{x}_k)\right)
$$
and
$$
\sum_{k=1}^{N}u(\bar{x}_k)u(x_k)=\sum_{k=1}^{N}u^2(\bar{x}_k)+\sum_{k=1}^{N}u(\bar{x}_k)\left(\ptu(x_k)-u(\bar{x}_k)\right),
$$
we can rewrite \eqref{st3} in the following form
\begin{multline}\label{st4} \frac d{dt} E_\eb(t)+b\|u_{xt}\|^2+\eb \|\px{u}\|+\eb(f(u),u)-a
\eb\|u\|^2-\eb\|\ptu\|^2 +\eb \mu
h\sum_{k=1}^{N}u^2(\bar{x}_k)=\\
-\mu
h\sum_{k=1}^{N}u(\bar{x}_k)\left(\ptu(x_k)-\ptu(\bar{x}_k)\right)-\eb
\mu h \sum_{k=1}^{N}u(\bar{x}_k)\left(u(x_k)-u(\bar{x}_k)\right),
\end{multline}
where
$$
E_\eb(t):=\frac12\|\ptu\|^2+\frac12(1+\eb
b)\|\px{u}\|^2-a\|u\|^2+(F(u),1)+\eb(u,\ptu)+\frac{\mu
h}2\sum_{k=1}^{N}u^2(\bar{x}_k).
$$
By using inequality \eqref{L2} we obtain
\begin{multline}\label{st8} E_\eb(t)\geq
\frac14\|\ptu\|^2+\left[\frac12(1+\eb
b)-2(a+\eb^2)h^2\right]\|\px{u}\|^2+\eb(F(u),1)+\\
\left[\frac {\mu
h}2-2(a+\eb^2)h\right]\sum_{k=1}^{N}u^2(\bar{x}_k). \end{multline}

Employing \eqref{L1} we get
\be\label{st9} \mu
h\sum_{k=1}^{N}u(\bar{x}_k)\left(\ptu(x_k)-\ptu(\bar{x}_k)\right)\leq
\frac{\eb\mu h}4\sum_{k=1}^{N}u^2(\bar{x}_k)+\frac{\mu
h^2}{\eb}\|\px \ptu\|^2, \ee
\be\label{st10} \mu
h\sum_{k=1}^{N}u(\bar{x}_k)\left(u(x_k)-u(\bar{x}_k)\right)\leq
\frac{\eb\mu h}4\sum_{k=1}^{N}u^2(\bar{x}_k)+\frac{\mu
h^2}{\eb}\|\px u\|^2.
 \ee
Now we choose $\eb =\frac{b\la_1}{2}$, use the Poincar\'{e} inequality,
\eqref{pnk}, and inequalities \eqref{st9} and
\eqref{st10} in \eqref{st4} and  obtain:
\begin{multline*}
 \frac d{dt} E_\eb(t)+\frac b2\|\px \ptu\|^2+\left(\eb -
 \frac4\eb\mu h^2\right) \|\px{u}\|^2+\eb(F(u),1)-a
\eb\|u\|^2+\\
\frac12\eb \mu h\sum_{k=1}^{N}u^2(\bar{x}_k)\leq 0.
\end{multline*}
We employ here inequality \eqref{L2} to obtain
\begin{multline}\label{st11}
 \frac d{dt} E_\eb(t)+\frac b2\|u_{xt}\|^2+\left(\eb -\frac4\eb\mu h^2-4a\eb h^2\right)
 \|\px{u}\|^2+\eb(F(u),1)+\\
 (\frac12\eb \mu h-2a\eb h)\sum_{k=1}^{N}u^2(\bar{x}_k)\leq 0.
\end{multline} It is not difficult to see that if $\mu$ is large enough such that
 \be\label{sc1} \mu
>4\left(a+\frac{\la_1^2b^2}{4}\right),
\ee and $N=\frac L h$ is large enough such that
\be\label{sc2}
\frac{\la_1b}{2}-2h^2\left(\frac\mu{\la_1b} -a\la_1 b\right)>0, \ee

\be\label{sc2a} \frac{b^2\la_1^2}{4}-a^2\la_1^2b^2h^2-\mu h^2>0, \ee
  then there exists $d_1>0$ such that
\be\label{sc3} E_\eb(t)\geq
d_1\left(\|\ptu\|^2+\|\px{u}\|^2\right),
 \ee
 and that there exists a positive number $\delta$ such that
\be\label{sc4} \frac b2\|u_{xt}\|^2+\left(\eb -\frac4\eb\mu
h^2\right) \|\px{u}\|^2+\eb(F(u),1)-a \eb\|u\|^2+ \frac12\eb \mu
h\sum_{k=1}^{N}u^2(\bar{x}_k)\geq \delta E_\eb(t). \ee
By virtue of
\eqref{sc4} we deduce from \eqref{st11} the inequality
$$
\frac d{dt} E_\eb(t)+\delta E_\eb(t)\leq 0.
$$
This inequality and inequality \eqref{sc3} imply the exponential
stabilization estimate
\be\label{est2} \|\ptu(t)\|^2+\|\px{u}(t)\|^2
\leq D_0 e^{-\delta t}.
\ee

Consequently we have proved the following:
\begin{theorem}\label{stT1}
Suppose that conditions \eqref{sc1}-\eqref{sc2a} are satisfied. Then all
solutions of the problem \eqref{st1}-\eqref{st2} tend to zero with
an exponential rate, as $t \to \infty$.
\end{theorem}

\begin{remark}By using similar arguments we can prove an analog to
Theorem \ref{stT1} for solutions of the semilinear pseudo-hyperbolic
equation \be\label{stR} \pt^2u-u_{xxtt}-\nu \px^2u-bu_{xxt}-a u+
f(u)=-\mu \sum_{k=1}^Nhu(\bar{x}_k)\delta (x-x_k), \ x \in (0,L),
t>0,
 \ee
 under the boundary conditions \eqref{st2} or under the periodic
 boundary conditions. Here $a,b,\nu$ are positive parameters,
  and $f$  satisfies  conditions \eqref{F},  $\mu$ and $N$ should be chosen appropriately large enough.
 \end{remark}
 \vspace{0.5cm}
\subsection{Feedback control employing  finitely many Fourier modes.}
The second problem we are going to study in this section is the
following feedback problem
\be\label{sw2}
\begin{cases}
 \pt^2u- \nu\Dx u -b\Dx
\ptu -a u+|u|^{p-2}u=  -\mu \sum\limits_{k=1}^N(u,w_k)w_k, \ x \in \Om,
t>0,\\
u(x,0)=u_0(x), \ \ptu(x,0)=u_1(x), \ x \in \Om,\\
 u=0, \  x \in \p \Om, \ t>0,
\end{cases}
\ee
where $\nu>0,b>0,a>0, p>2$ are given
parameters.\\
The following theorem guarantees the exponential feedback stabilization
of solutions of problem \eqref{sw2}
\begin{theorem}\label{ThSt}
Suppose that $\mu$ is large enough such that
\be\label{strc1}
\mu >2a+\frac14\delta_0\la_1b,
\ee
and $N$ is large enough such that
\be\label{strc2}
\nu\geq\left(2a+\frac14b\la_1\delta_0\right)\la_{N+1}^{-1},
\ee
where
\be\label{del0}
\delta_0:=\frac{b\la_1\nu}{2\nu+b^2\la_1}.
\ee
Then the solution of \eqref{sw2} satisfies the following exponential decay estimate:
\be\label{STest}
\|\ptu(t)\|^2+\|\Nx u(t)\|^2+\|u(t)\|_{L^p(\Om)}^p\leq E_0e^{-\delta_0t},\,\, \mbox{for all}\, t >0.
\ee
 with a constant  $E_0$ depending on initial data.
\end{theorem}
\begin{proof} The proof of this theorem is similar to the proof of Theorem
\ref{TF}.\\
The energy equality in this case has the form
\begin{multline}\label{ST2}
\frac d{dt} E_\eb(t) + b\|\Nx \ptu(t)\|^2-\eb \|\ptu(t)\|^2\\
+\eb\nu\|\Nx
u(t)\|^2+\eb\int_\Om|u(x,t)|^pdx-\eb a \|u\|^2+\eb\mu
\sum\limits_{k=1}^N(u(t),w_k)^2=0,
\end{multline}
where $\eb=\frac{b\la_1}2,$  and
$$
E_\eb(t) :=\frac12\|\ptu\|^2+\frac12(\nu +\eb b)\|\Nx u\|^2-\frac
a2\|u\|^2+\frac1p\|u\|_{L^p(\Om)}^p+\frac \mu2\sum\limits_{k=1}^n(u,w_k)^2+
\eb(u,\ptu).
$$
Employing the Young inequality \eqref{yng}, the Poincar\'{e} inequality \eqref{pnk},  the Poincar\'{e}-like inequality \eqref{pnk}, and the conditions \eqref{strc1} and \eqref{strc2} we get:
\begin{multline}\label{eeb}
E_\eb(t)\geq \frac14\|\ptu\|^2+\frac12\left( \nu +\frac{b^2\la_1}2\right)\|\Nx u\|^2-\left(\frac a2 +\frac{b^2\la_1^2}4\right)\|u\|^2
+\frac1p\|u\|_{L^p(\Om)}^p\\
+\frac \mu2\sum\limits_{k=1}^n(u,w_k)^2\geq  \frac14\|\ptu\|^2+\frac\nu2\|\Nx u\|^2-\frac a2\|u\|^2+
\frac \mu2\sum\limits_{k=1}^n(u,w_k)^2+\frac1p\|u\|_{L^p(\Om)}^p\\
\geq \frac14\|\ptu\|^2+\frac\nu4\|\Nx u\|^2+\left( \frac \nu4-\frac a2\la_{N+1}^{-1}\right)\|\Nx u\|^2+\left(\frac \mu2-\frac a2\right)
\sum\limits_{k=1}^n(u,w_k)^2+  \frac1p\|u\|_{L^p(\Om)}^p\\
\geq \frac14\|\ptu\|^2+\frac \nu4 \|\Nx u\|^2+\frac1p\|u\|_{L^p(\Om)}^p.
\end{multline}

Let $\delta\in (0,\eb)$, be another parameter to be chosen below. Then employing the Poincar\'{e} inequality, and the fact that $\delta <\eb=\la_1b/2$, and $p>2,$ we  obtain from \eqref{ST2}
\begin{multline}\label{bd1}
\frac d{dt}E_\eb(t)+\delta E_\eb(t)+\left(\frac b2-\frac{\delta}{2\la_1}\right)\|\Nx \ptu\|^2\\
+\left(\frac{b\la_1\nu}2-\frac{\delta}2 (\nu+\frac{b^2\la_1}2)\right)\|\Nx u\|^2-\frac{ab\la_1}4\|u\|^2+
\left(\frac{b\la_1}2-\frac\delta p\right)\|u\|_{L^p(\Om)}^p\\
+\frac\mu2\left( b\la_1-\delta\right)\sum\limits_{k=1}^n(u,w_k)^2-\frac{\delta \la_1b}2(u,\ptu)\leq 0.
\end{multline}
By using the inequality
$$
\frac12 \delta\la_1b|(u,\ptu)|\leq \frac12 \delta\la_1^ {\frac12}b\|\Nx \ptu\|\|u\|\leq       \frac \delta{2\la_1}\|\Nx \ptu\|^2+\frac\delta 8 \la_1^2b^2\|u\|^2,
$$
and assuming here that $\delta\in (0,\eb)$, we get
\begin{multline}\label{strc3}
\frac d{dt}E_\eb(t)+\delta E_\eb(t)+\left(\frac b2-\frac\delta{\la_1}\right)\|\Nx \ptu\|^2
+\frac12\left[\nu b\la_1 -\delta(\nu +\frac{b^2\la_1}2)\right]\|\Nx u\|^2\\
-\frac12\left(ab\la_1+\frac \delta 8 \la_1^2b^2\right)\|u\|^2+\frac \mu2(b\la_1-\delta)\sum\limits_{k=1}^n(u,w_k)^2.
\end{multline}

We choose in the last inequality $\delta=\delta_0$, defined in \eqref{del0},
and obtain the inequality
\be\label{strc3a}
\frac d{dt}E_\eb(t)+\delta_0 E_\eb(t)
+\frac14\nu b\la_1 \|\Nx u\|^2
-\frac{b\la_1}2\left(a+\frac \delta 8 \la_1b\right)\|u\|^2+\frac \mu 4b\la_1\sum\limits_{k=1}^n(u,w_k)^2.
\ee
Finally, employing in \eqref{strc3a}  inequality \eqref{QN}, and the conditions \eqref{strc1} and \eqref{strc2} we obtain the desired inequality
$$
\frac d{dt}E_\eb(t)+\delta_0 E_\eb(t)\leq 0.
$$
Thanks to \eqref{eeb}, this inequality implies the desired decay estimate \eqref{STest}.
\end{proof}
\begin{remark}\label{Bsq}
In a similar way we can prove exponential stabilization of solutions to  strongly damped Boussinesq equation,
with homogeneous Dirichlet boundary conditions
\be\label{bouss}
\begin{cases}
\p_t^2 u-\nu \p_x^4 u -b\p_x^2\ptu+\p_x^2\left(au -|u|^{p-2}u\right)=-\mu w, \ x\in (0,L), t>0,\\
u(0,t)=u(L,t)=\p_x^2u(0,t)=\p_x^2u(L,t)=0, \ \ t>0,
\end{cases}
\ee
where $a, \nu, b$ are given positive parameters, and $w$ is a controller of the
$$
w=\sum\limits_{k=1}^N \la_k(u,w_k)w_k,
\ \
\la_k=\frac{k^2\pi^2}{L^2}, \ \ w_k(x)=\sin\frac{k\pi}{L}.
$$
\end{remark}
Here also we can find $N$ and $\mu$ large enough such that
$$
\|\ptu(t)\|^2_{H^{-1}(0,L)}+\|\p_x^2u(t)\|^2\leq E_0e^{-\delta t}
$$
with some $\delta>0$, depending on parameters of the problem, i.e. $\nu, b, p$ and $L$.
 \begin{remark}\label{existence} It is worth noting that the
 estimates we obtained in Theorem \ref{TF} and Theorem \ref{ThSt} are valid
 for weak solutions of the corresponding problems from the class of
 functions such that
 $$
\ptu\in L^\infty(\R^+; L^2(\Om)), \ u\in L^\infty(\R^+; H_0^1(\Om)\cap
L^p(\Om)).
 $$
Existence of a weak solution for each of these problems, as well as
justification of our estimates  can be done by using the Galerkin
method (see e.g. \cite{Lions}). We would like also to note that the
feedback stabilization estimate for  problem
\eqref{fo1}-\eqref{fo2} is established, for arbitrary $p>2$, without
any restrictions on the spatial dimension of the domain $\Om$. As far as we
know, even uniqueness of a weak solution in the case $p>5, \Omega
\subset \R^3$, is an open problem.
 \end{remark}
 \begin{remark}\label{defu}
It is also worth mentioning that the feedback stabilization of nonlinear damped wave equation, nonlinear strongly damped wave equation, nonlinear wave equation with nonlinear damping term  with controllers involving finitely many parameters, can be shown by employing the concept of general determining functionals (projections) introduced in \cite{CJT},\cite{Cock97}, then exploited and developed for the study of dissipative wave equations  \cite{Ch}.\\
 Recently it was shown in \cite{FoJoKrTi1}, \cite{FoJoKrTi2} and \cite{JoSaTi} that the approach of a new feedback controlling of dissipative
 PDEs using finitely many determining parameters can be used to show that the global (in time) dynamics of the 2D Navier - Stokes equations, and of that of the 1D damped driven nonlinear Schr\"{o}dinger equation
 $$
 i\ptu -\px^2 u +i\gamma u +|u|^2u=f, \ \ b>0,
 $$
 can be embedded in an infinite-dimensional dynamical system induced by an ordinary differential equation,  called {\it determining form}, governed by a global Lipschitz vector field.\\
 The existence of determining form for  the long-time dynamics of nonlinear damped wave equation and nonlinear strongly damped wave equation is a subject of a future research.\\
 \end{remark}
 \noindent {\bf Acknowledgements.}  The work of V.~K.~Kalantarov was
supported in part by The Scientific and Research Council of Turkey,
grant no. 112T934. The work of E.S.~Titi was supported in part by the NSF grants DMS--1109640 and DMS--1109645.

\end{document}